\documentclass[12pt]{article}

\usepackage{epsfig,amsmath,amsfonts,amssymb,latexsym,color,amsthm,hhline,graphicx}

\newtheorem{theorem}{Theorem}[section]
\newtheorem{prop}{Proposition}[section]

\newtheorem{corr}{Corollary}[section]
\newtheorem{conj}{Conjecture}[section]
\newcommand{\E}{{\mathbb E}}

\newcommand {\PP}{{\mathbb P}}
\newcommand{\sss}{\scriptscriptstyle}

\begin{document}

\title{A spatial epidemic model with site contamination}\parskip=5pt plus1pt minus1pt \parindent=0pt
\author{Tom Britton\thanks{Department of Mathematics, Stockholm University; {\tt tomb@math.su.se}} \and Maria Deijfen\thanks{Department of Mathematics, Stockholm University; {\tt mia@math.su.se}} \and Fabio Lopes\thanks{Center for Mathematical Modeling, Universidad de Chile; {\tt flopes@dim.uchile.cl}}}
\date{May 2017}
\maketitle

\begin{abstract}
We introduce the effect of site contamination in a model for spatial epidemic spread and show that the presence of site contamination may have a strict effect on the model in the sense that it can make an otherwise subcritical process supercritical. Each site on $\mathbb{Z}^d$ is independently assigned a random number of particles and these then perform random walks restricted to bounded regions around their home locations. At time 0, the origin is infected along with all its particles. The infection then spread in that an infected particle that jumps to a new site causes the site along with all particles located there to be infected. Also, a healthy particle that jumps to a site where infection is presents, either in that the site is infected or in the presence of infected particles, becomes infected. Particles and sites recover at rate $\lambda$ and $\gamma$, respectively, and then become susceptible to the infection again. We show that, for each given value of $\lambda$, there is a positive probability that the infection survives indefinitely if $\gamma$ is sufficiently small, and that, for each given value of $\gamma$, the infection dies out almost surely if $\lambda$ is large enough. Several open problems and modifications of the model are discussed, and some natural conjectures are supported by simulations.

\noindent
\vspace{0.3cm}


\vspace{0.2cm}

\end{abstract}

\section{Introduction}

Stochastic epidemic models describe the spread of an infection transmitted in some random way via contacts between individuals; see e.g.\ \cite{AB} for an introduction. The simplest models deal with unstructured populations, but in many situations it is natural to incorporate spatial structure, so that the individuals have positions in space. The most studied case in the mathematical literature is when the individuals are represented by the sites of the $\mathbb{Z}^d$ lattice; see \cite{AS,vdBGS,DN,Kul}. We are interested in modeling a situation where individuals do not have a fixed position, but move in space. They may infect each other upon direct contact, but may also be infected via contaminated locations. In particular for macroparasite diseases among domestic animal populations this is often the case \cite{H_15}, but also human diseases may spread through contaminated locations such as door handles and toilet facilities.  We formulate a simple model incorporating this phenomenon and demonstrate that the presence of site contamination may have an impact on the epidemic spread. Specifically, we show that site contamination can cause a subcritical model -- that is, a model where the infection cannot survive indefinitely -- to become supercritical. Conversely, measures aimed at controlling site contamination has the potential to make a supercritical model subcritical. We also formulate a number of open problems for our model, and other versions of it.

The model that we will analyze describes the evolution of an infection in a population of moving particles on $\mathbb{Z}^d$. To start with, for each $x\in\mathbb{Z}^d$, a number $M_x$ of healthy particles are placed at $x$, where $\{M_x\}$ are i.i.d.\ random variables. Each particle is marked with its home location, which is the site where it is initially placed. Furthermore, each particle has an accessible region associated with it, denoted by $\mathcal{R}_x$ and consisting of all sites within $L_1$-distance at most $k<\infty$ from its home location. A particle can move in its accessible region, but not beyond that. Specifically, starting at time 0, the particles perform independent random walks with jump rate $\beta>0$ in their respective accessible regions, that is, after an exponentially distributed time with mean $\beta^{-1}$, a particle at $y$ with home location $x$ jumps to a randomly chosen site in $\mathcal{N}_y\cap \mathcal{R}_x$, where $\mathcal{N}_y$ denotes the set of nearest neighbors of $y$. Also at time 0, the origin and all its  particles are infected. The infection then evolves in continuous time according to the following rules. If an infected particle jumps to a new site, then the site becomes contaminated and all other particles located there become infected. If a healthy particle jumps to a contaminated site or to a site where there are infected particles present, then it becomes infected. An infected particle recovers at rate $\lambda$ and a contaminated site is cleared at rate $\gamma$, independently of everything else. Recovered particles and cleared sites become susceptible to the infection again, and may be re-infected. The model is hence of SIS-type for both particles and sites. Note that particles/sites are only infected when a jump occurs. A particle that recovers is hence not immediately re-infected by the site where it is located in case this is infected, or by other infected particles already present at the same site. Similarly, a site that is cleared is not immediately re-infected by infected particles located at the site.

We remark that the model could of course also be studied with $k=\infty$, so that the particles perform independent unrestricted random walks on $\mathbb{Z}^d$. This case however is likely to be considerably more complicated to analyze. The model has then been analyzed without site contamination by Kesten and Sidoravicius \cite{KS_06}; see below for further comments.

The model has three parameters $\beta$, $\lambda$ and $\gamma$ describing the time dynamics, and $M$ describing the (random) initial number of particles at a site. Without loss of generality (time-scaling) we set $\beta=1$ from now on. We will be interested in how the possibility for the infection to survive indefinitely is affected as $\lambda$ and/or $\gamma$ change.  Write $S_{\lambda,\gamma}$ for the event that there are infected sites or particles at infinitely large times. Clearly, if $\lambda=0$ or $\gamma=0$, so that particles never recover or sites are never cleared, then $\PP(S_{\lambda,\gamma})=1$.  If $\lambda=\infty$, so that particles recover instantaneously, then $\PP(S_{\lambda,\gamma})=0$ (provided $\gamma>0$), while $\gamma=\infty$ corresponds to a model without site contamination, where the value of $\PP(S_{\lambda,\gamma})$ is not a priori clear. We will hence assume that $\lambda\in(0,\infty)$ and $\gamma\in(0,\infty]$.

We first observe that, for a fixed value of $\gamma$ ($\lambda$), the probability of infinite survival is non-decreasing as $\lambda$ ($\gamma$) decreases, that is, slower recovery for particles/sites makes it easier (or at least not harder) for the infection to survive.

\begin{prop}[Monotonicity in $\lambda$ and $\gamma$]\label{prop:monoton} If $0<\gamma_0<\gamma_1<\infty$, then $\PP(S_{\lambda,\gamma_0})\geq \PP(S_{\lambda,\gamma_1})$ for any fixed value of $\lambda$. Similarly, if $0<\lambda_0<\lambda_1<\infty$, then $\PP(S_{\lambda_0,\gamma})\geq \PP(S_{\lambda_1,\gamma})$ for any fixed value of $\gamma$.
\end{prop}

According to Proposition \ref{prop:monoton}, a larger element of site contamination (quantified by the site clearance rate $\gamma$) makes it easier for the infection to survive indefinitely, but does it make it \emph{strictly} easier? Our two main results give a partial answer to this. The first result states that, for any value of the particle recovery rate $\lambda$, if the clearance rate of the sites is small, then the infection has a strictly positive probability of surviving indefinitely. The result is proved by comparison with a site percolation process, and is therefore only valid for $d\geq 2$, but we conjecture that it is true also for $d=1$; see below for further comments. Here $p_c^{\sss\rm{site}}$ denotes the critical value for site percolation.

\begin{theorem}[Supercriticality for small $\gamma$]\label{th:super}
Fix $d\geq 2$, $k\geq 1$ and $\lambda\in(0,\infty)$, and assume that $\PP(M=0)<p_c^{\sss\rm{site}}$. Then $\PP(S_{\lambda,\gamma})>0$ if $\gamma>0$ is sufficiently small.
\end{theorem}

The assumption on the distribution of $M$ can presumably be weakened, but some condition on $\PP(M=0)$ is indeed necessary to guarantee that the particle configuration is not too sparse. The second result states that, for any value of the site clearance rate $\gamma$, almost surely the infection dies out in finite time when the particle infectious period is short enough (that is, when the recovery rate is large enough). The result is proved under the assumption that the number of particles per site is almost surely bounded. We believe that this assumption can be weakened, possibly to the extent that no assumption on the tail of $M$ is needed.

\begin{theorem}[Subcriticality for large $\lambda$]\label{th:sub}
Fix $d\geq 1$, $k\geq 1$ and $\gamma\in(0,\infty]$, and assume that $M\leq \bar{m}$ a.s.\ for some $\bar{m}<\infty$. Then $\PP(S_{\lambda,\gamma})=0$ if $\lambda$ is sufficiently large.
\end{theorem}

Intuitively, the probability that infection is transferred from an infected site to neighboring sites can be made arbitrarily large by increasing the site infection time (that is, by decreasing the clearance rate). The particles that have access to the site will then have time to return to the infected site many times and get re-infected, and each visit gives rise to an opportunity to pass the infection on. This explains Theorem \ref{th:super}. On the other hand, the probability of infection transfer can be made arbitrarily small by letting the particles recover very fast. Even if particles return to an infected site many times, the probability that they pass the infection on will then be small. This explains Theorem \ref{th:sub}.

Combining Theorem \ref{th:super} and \ref{th:sub}, yields that a larger element of site contamination can make a strict difference in the sense that it can make a subcritical model supercritical.

\begin{corr}\label{corr:super}
Fix $d\geq 2$, $k\geq 1$ and assume that $\PP(M=0)<p_c^{\sss\rm{site}}$ and $M\leq \bar{m}$ a.s.\ for some $\bar{m}<\infty$. Let $\lambda_0$ and $\gamma_0$ be such that $\PP(S_{\lambda_0,\gamma_0})=0$. Then there exist $\gamma_1\in(0,\gamma_0)$ such that $\PP(S_{\lambda_0,\gamma})>0$ for $\gamma<\gamma_1$.
\end{corr}

We remark at this point that the proof of Theorem \ref{th:super} is valid also for a version of the model where infection is only transferred to particles upon contact with a contaminated site, that is, a particle does not get infected by having contact with another infected particle, but only by jumping onto a contaminated site. This may be more realistic for some types of cattle epidemics, e.g.\ udder disease among cows. This version of the model is stochastically smaller than our original model, so Theorem \ref{th:sub} (and thereby also Corollary \ref{corr:super}) remains trivially true.

\subsubsection*{Open problems and related work}

Our ambition with this paper is to introduce the effect of site contamination in a spatial epidemic model and to prove some initial results using relatively simple methods. Several issues remain to deal with. According to Corollary \ref{corr:super}, for a given value of $\lambda$, the model can be made supercritical by picking $\gamma$ small. Is it also true that, if we pick $\gamma$ large enough for a given value of $\lambda$, the model becomes subcritical? We conjecture that this is not true in general, but depends on the fixed value of $\lambda$. To understand this, we first note that it is reasonable to expect that Theorem \ref{th:sub} has a counterpart saying that $\PP(S_{\lambda,\gamma})>0$ if $\lambda$ is sufficiently small. Indeed, if particles remain infected for a very long time, then the infection should be able to survive. Combining this with Proposition \ref{prop:monoton} yields that, for a given value of $\gamma$, there is a phase transition in $\lambda$. Furthermore, we conjecture that the critical value is strictly decreasing as $\gamma$ increases, that is, the supercritical region becomes smaller if sites recover faster.

\begin{conj}[Phase transition in $\lambda$]\label{conj:pt_lambda}
Fix $d\geq 1$, $k\geq 1$ and $\gamma\in(0,\infty]$. Under suitable assumptions on the distribution of $M$, there exists a critical value $\lambda_c^{\gamma}\in(0,\infty)$ such that $\PP(S_{\lambda,\gamma})>0$ if $\lambda<\lambda_c^{\gamma}$ but $\PP(S_{\lambda,\gamma})=0$ if $\lambda>\lambda_c^{\gamma}$. Furthermore, if $\gamma_0<\gamma_1$, then $\lambda_c^{\gamma_0} >\lambda_c^{\gamma_1}$.
\end{conj}

Here ``suitable assumptions on the distribution of $M$'' could for instance be the ones in Corollary \ref{corr:super} but, as indicated, these could presumably be weakened. Let $\lambda_c^{\infty}$ denote the critical value in a model without site contamination. It follows from Conjecture \ref{conj:pt_lambda} and Proposition \ref{prop:monoton} that, regardless of the value of $\gamma$, the model is supercritical for $\lambda<\lambda_c^{\infty}$, since site infection helps the infection. A small value of $\lambda$ is hence sufficient in itself to make the model supercritical. For $\lambda>\lambda_c^{\infty}$ on the other hand, it is reasonable to expect that there is a non-trivial phase transition in $\gamma$. If $\gamma$ is large, infected sites will then be cleared very fast and should not be able to change the subcriticality of the model without site infection.

\begin{conj}[Phase transition in $\gamma$]\label{conj:pt_gamma}
Fix $d\geq 1$, $k\geq 1$ and $\lambda\in(0,\infty)$, and impose suitable assumptions on the distribution of $M$. If $\lambda<\lambda_c^{\infty}$, then $\PP(S_{\lambda,\gamma})>0$ for all values of $\gamma\in(0,\infty]$. If $\lambda>\lambda_c^{\infty}$, then there exists a critical value $\gamma_c^{\lambda}\in(0,\infty)$ such that $\PP(S_{\lambda,\gamma})>0$ for $\gamma<\gamma_c^{\lambda}$, but $\PP(S_{\lambda,\gamma})=0$ for $\gamma>\gamma_c^{\lambda}$.
\end{conj}

To summarize, by tuning $\lambda$ we can control the behavior of the model for any given value of $\gamma$ while, by tuning $\gamma$, we may not be able to make the model subcritical. Indeed, the infection is driven by contacts between particles and by contacts between particles and sites. The parameter $\lambda$ can control infection arising from both theses effects, while $\gamma$ can only control the latter effect.

\begin{figure}
\label{fig}
\centering
\includegraphics[height=7cm,width=7cm]{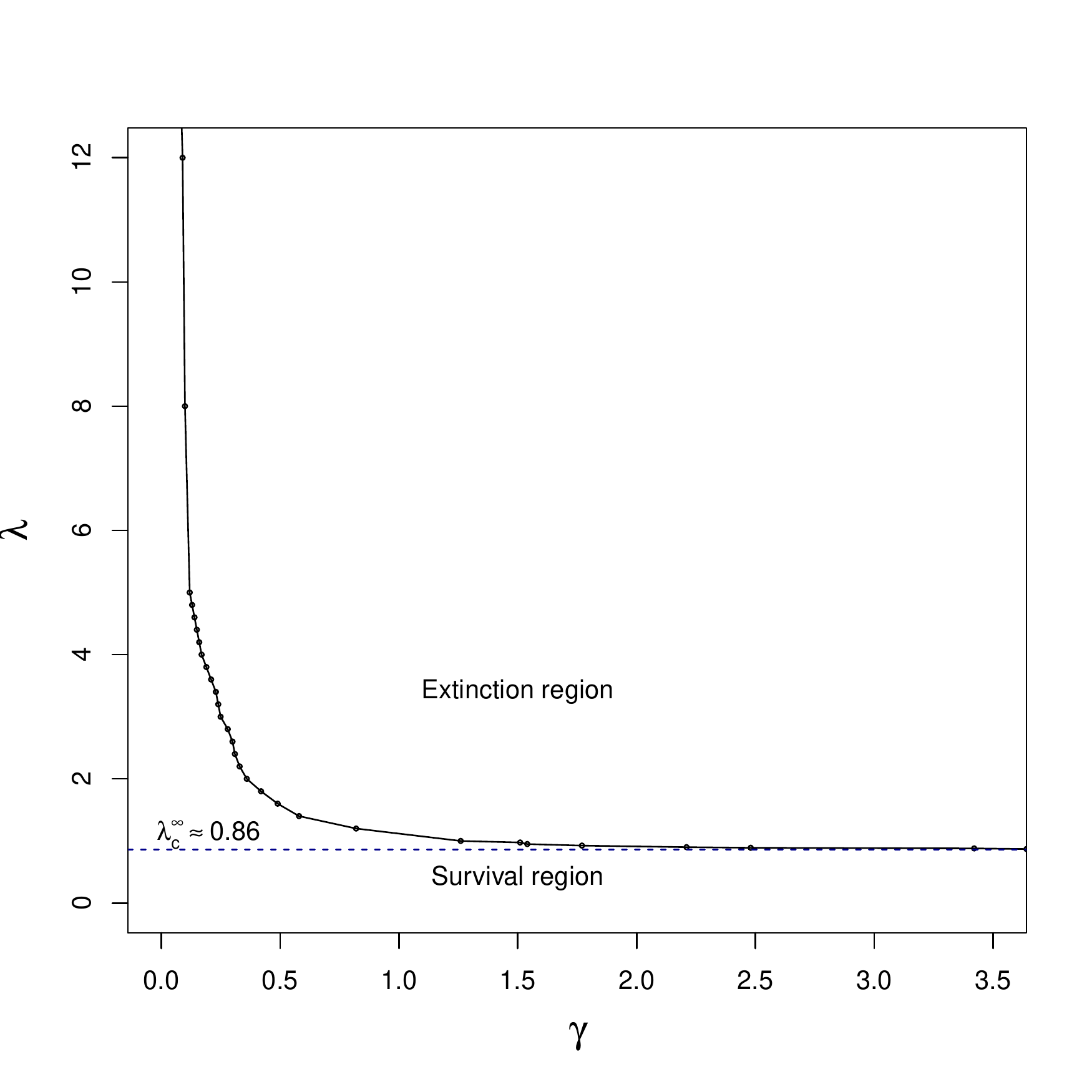}
\includegraphics[height=7cm,width=7cm]{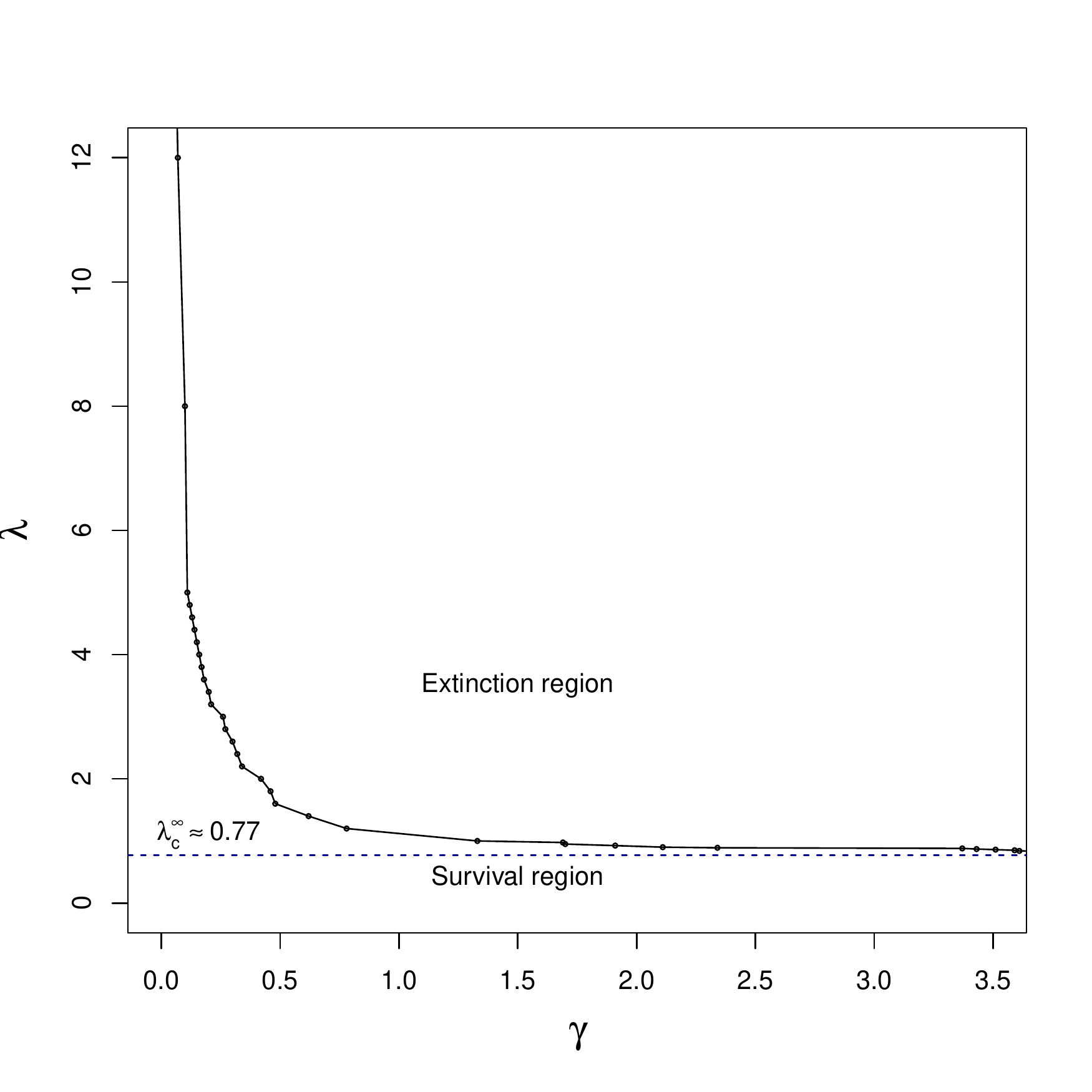}
\caption{Estimates of the phase diagram on $\mathbb{Z}^2$ with $M=1$ and $k=1$. In both pictures we have $L=30$ and $K=100000$. In the left-hand picture, we have $S=30$ and the estimate of $\lambda_c^{\infty}$ is approximately $.86$ while, in the right-hand figure, we have $S=20$ and the estimate of $\lambda_c^{\infty}$ is approximately $.77$.}
\end{figure}

We illustrate the conjectures through simulations of our process with $M=1$ and $k=1$ on a square lattice with $L\times L$ sites under periodic boundary conditions. The approximation we obtain for the phase diagram in this case supports the above conjectures; see Figure \ref{fig}. First, in order to obtain an estimate for $\lambda_c^{\infty}$, we consider the process without site infection (that is, with $\gamma=\infty$). For fixed $K,S>0$ and a given initial $\lambda_0>0$, we run up to $S$ simulations of the process with recovery rate $\lambda_0$. Each simulation consists of up to $K$ steps for the process, where each step is either a particle jump, a particle recovery or a site clearance. The simulations are run sequently until the first time we observe at least one infected particle at the $K$-th step. If this occurs, we declare that the process survives with positive probability for that value of $\lambda_0$. If this does not occur in the $S$ simulations, we repeat the procedure with smaller values of $\lambda_0$, until we find a first value of  $\lambda_0$ for which the simulated process has at least one infected particle at the $K$-th step.  Then, for a given set of recovery rates $\{\lambda_1, \ldots, \lambda_m \}$, $m\geq 1$, all greater than the estimate of $\lambda_c^{\infty}$, we run a similar procedure to estimate the values of $\{\gamma_c^{\lambda_i}\}_{i=1}^{m}$. The picture of the phase diagram obtained from this procedure is not exact, since changes in the settings -- specifically, in the values of $L$, $K$ and $S$ -- affect the precise values of the estimated thresholds. However, the results obtained for different settings all exhibit a global behavior which supports Conjecture \ref{conj:pt_lambda} and \ref{conj:pt_gamma}; see Figure \ref{fig} for two examples.

To confirm Conjecture \ref{conj:pt_lambda}, it remains to show that $\PP(S_{\lambda,\gamma})>0$ if $\lambda$ is small enough. This could perhaps be done by comparison with a site percolation process along the same lines as in the proof of Theorem \ref{th:super}. Relying on direct infections between particles would however require a dependent percolation process; see the remark following the proof of Theorem \ref{th:super}. Proving strict inequalities for the critical value is presumably harder, but may be possible using enhancement techniques (see \cite[Section 3.3]{Grim}) adapted to interacting particle systems. As mentioned, we are convinced that Theorem \ref{th:super} is true also in $d=1$. This might be possible to prove by modifying techniques developed for the contact process and other interacting particle systems; see \cite{DBook,DBook2}.

We have analyzed a model where the range of the particles is bounded by $k<\infty$. Taking $k=\infty$ gives a model where the particles perform independent unrestricted random walks on $\mathbb{Z}^d$. Without site contamination, this model was first formulated by Spitzer in the 1970s and it has turned out to be technically very difficult to handle. For a Poisson distributed initial number of particles per site, this model was shown in an extensive paper by Kesten and Sidoravicius \cite{KS_06} to have a phase transition in $\lambda$, and versions of the model without recovery have been considered in \cite{KS_03,KS_05, KS_08}. The fact that our particles can move only in a bounded region gives better control of the particle configuration and greatly simplifies the analysis. It is of course a natural project to try to prove analogues of Conjecture \ref{conj:pt_lambda} and \ref{conj:pt_gamma} for a model with unrestricted movements.

As for related models, we mention the activated random walk model, which is similar to the model with $k=\infty$ considered in \cite{KS_06}, but where only the infected particles move; see e.g.\ \cite{DRS,RS,ST}. Furthermore, a situation where only the infected particles move and, in addition, recovered particles are removed from the system is described by the frog model; see \cite{frogs_shape, pt_frogs, frogs_surv}. The fact that only infected particles move makes the analysis simpler compared to the case when also the healthy particles move. However, such a simplification is not possible in our setup, since for site contamination to play a role clearly also healthy particles need to move.

\section{Proofs}

In this section we prove Theorem \ref{th:super} by bounding it with a site percolation process and \ref{th:sub} by dominating it with a branching process. We first prove Proposition \ref{prop:monoton}, which is immediate from the natural coupling of processes with different parameter values.

\begin{proof}[Proof of Proposition \ref{prop:monoton}.]
Let $S_{\lambda,\gamma}(t)$ denote the event that there are infected particles or sites at time $t$. Clearly $S_{\lambda,\gamma}(s)\supset S_{\lambda,\gamma} (t)$ for $s\leq t$ and hence
\begin{equation}\label{eq:inf_surv_limit}
\PP(S_{\lambda,\gamma})=\lim_{t\to\infty}\PP(S_{\lambda,\gamma}(t)).
\end{equation}
Fix $\lambda$ and $0\leq \gamma_0\leq \gamma_1$. The processes with parameter values $(\lambda,\gamma_0)$ and $(\lambda,\gamma_1)$, respectively, can be coupled in such a way that, if a particle/site is infected at time $t$ in the process with site clearance rate $\gamma_1$, then it is infected at time $t$ also in the process with rate $\gamma_0$: Let each particle independently be equipped with a Poisson process with rate $\lambda$, specifying the possible recovery times of the particle, that is, if the particle is infected at time $t-\varepsilon$ and there is a Poisson event at time $t$, then the particle goes back to being susceptible at time $t$. Similarly, let each site be equipped with a Poisson process with rate $\gamma_1$, specifying the possible clearance times for the sites in the epidemic process with rate $\gamma_1$. For the epidemic process with rate $\gamma_0$, we let the site clearance times be specified by thinned versions of the rate $\gamma_1$ Poisson processes, where each Poisson point is removed independently with probability $\gamma_0/\gamma_1$. If we let the particles move according to the same random walks in the two processes, it is not hard to see that, if a particle/site is infected at time $t$ in the process with rate $\gamma_1$, then it is infected at time $t$ also in the process with rate $\gamma_0$. This implies that $\PP(S_{\lambda,\gamma_0}(t))\geq \PP(S_{\lambda,\gamma_1}(t))$ for all $t$ and it follows from \eqref{eq:inf_surv_limit} that $\PP(S_{\lambda,\gamma_0})\geq \PP(S_{\lambda,\gamma_1})$. The second claim in Proposition \ref{prop:monoton} follows analogously by coupling the Poisson processes governing the particle recovery times.
\end{proof}

\begin{proof}[Proof of Theorem \ref{th:super}.] We give the proof for $d=2$, but note along the way that it is easy to generalize to $d\geq 3$. We consider a process that evolves according to the same rules as our epidemic process, but with fewer particles. Specifically, in the initial configuration, particles are placed only at points $\mathbb{Z}^2_k= \{(ik,jk):i\in\mathbb{Z} \textrm{ even}, j\in\mathbb{Z} \textrm{ odd}\}$; see Figure \ref{fig:sparse_grid}. Removing particles clearly leads to a model where it is harder for the infection to survive, and we are therefore done if we show that infinite survival has positive probability starting from this more sparse configuration.

\begin{figure}
\centering
\includegraphics[width=.6\textwidth]{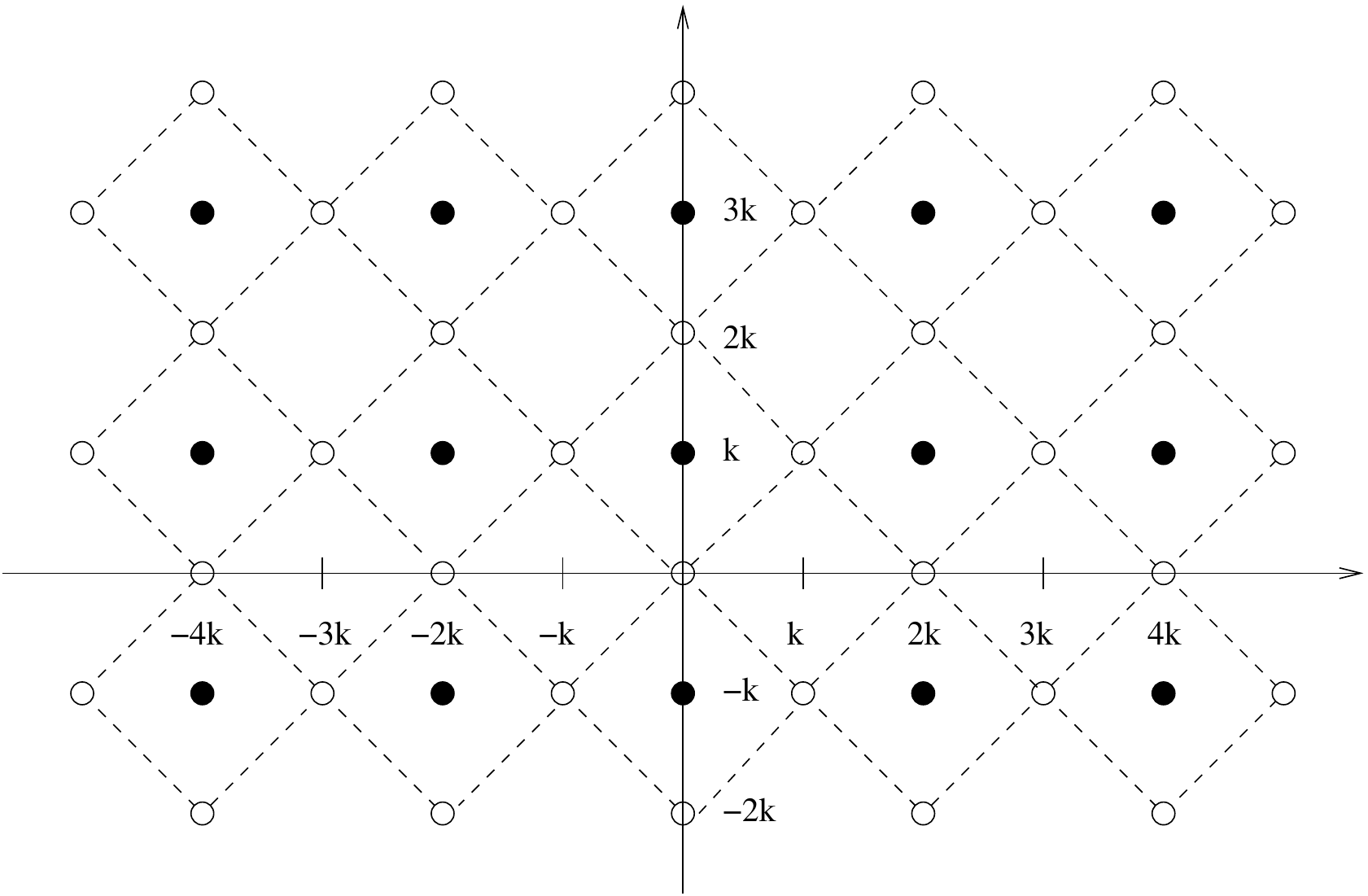}
\caption{Picture of $\mathbb{Z}^2_k$. Black dots represent center sites, where particles have their home locations, and white dots represent corner sites of accessible regions.}\label{fig:sparse_grid}
\end{figure}

The accessible region $\mathcal{R}_z$ of a particle $\xi_z$ with home location $z=(ik,jk)\in\mathbb{Z}^2_k$ consists of the point $z$, which we shall refer to as the center site, along with all sites at $L_1$-distance at most $k$ from $z$. We shall refer to the sites $(ik\pm k,jk\pm k)$ as corner sites of $\mathcal{R}_z$; see Figure \ref{fig:sparse_grid}. For each region $\mathcal{R}_z$, $z\in\mathbb{Z}^2_k$, we say that it is activated when the infection first enters the region. This happens in that a particle $\xi_{(ik\pm 2k,jk\pm 2k)}$ of one of the adjacent regions $\mathcal{R}_{(ik\pm 2k,jk\pm 2k)}$ visits a corner site that is shared with $\mathcal{R}_z$, causing that site to be contaminated. We declare an activated region $\mathcal{R}_z$, $z\in\mathbb{Z}^2_k$, to be open if (i) $M_z\geq 1$, so that there is at least one particle with $z$ as home location, and (ii) from the moment when the region is activated, one of its particles $\xi_z$ manages to transfer the infection from the activating corner site, denoted by $z_{\sss act}$, to all other corner sites in $\mathcal{R}_z$ before $z_{\sss act}$ recovers. Clearly this causes all other adjacent regions to be activated in case they have not already been so, and we conclude that the infection for sure survives indefinitely if the origin belongs to an infinite cluster of open regions.

As for the probability that an activated region $\mathcal{R}_z$ is open, assume that $M_z\geq 1$ and let $A_{\gamma}$ denote the event that a given particle $\xi_z$ visits $z_{\sss act}$ before $z_{\sss act}$ is cleared (counting from the time of activation). The probability $\PP(A_{\gamma})$ becomes smaller if $\xi_z$ is located far away from $z_{\sss act}$ at the time of activation, but it is not hard to see that it can be uniformly bounded from below (regardless of the position of $\xi_z$ at the time of activation) by a bound that can be made arbitrarily close to 1 by decreasing the site recovery rate $\gamma$. When visiting $z_{\sss act}$, the particle $\xi_z$ then becomes infected. Once $\xi_z$ is positioned at $z_{\sss act}$ it will make excursions from  $z_{\sss act}$, eventually returning to $z_{\sss act}$ and, if $z_{\sss act}$ is still contaminated when $\xi_z$ returns, then $\xi_z$ becomes infected again (in case it has recovered). Let $B_{\lambda,\gamma}$ denote the event that, before $z_{\sss act}$ is cleared, the particle $\xi_z$ makes an excursion from $z_{\sss act}$ that visits all other corner sites of $\mathcal{R}_z$ and that, in addition, $\xi_z$ remains infected during the whole excursion. The probability that a given excursion visits all sites in $\mathcal{R}_z\setminus z_{\sss act}$ while $\xi_z$ remains infected may well be very small, in particular if the particle recovery rate $\lambda$ is large. However, by decreasing $\gamma$ so that $z_{\sss act}$ stays contaminated for a very long time, we can give $\xi_z$ very many attempts to succeed with such an excursion. Hence, for a given value of $\lambda$, the probability $\PP(B_{\lambda,\gamma})$ can be made arbitrarily close to 1 by decreasing $\gamma$. Furthermore, clearly the event $A_\gamma\cap B_{\lambda,\gamma}$ implies that the region is declared open.

To summarize, we can bound the probability that an activated region is open by $p_{\lambda,\gamma}=\PP(M\geq 1)\tilde{p}_{\lambda,\gamma}$ where, for a given value of $\lambda$, the probability $\tilde{p}_{\lambda,\gamma}$ can be made arbitrarily close to 1 by decreasing $\gamma$ and, by assumption, $\PP(M\geq 1)>p_c^{\sss\rm{site}}$. Furthermore, if a region is declared open or not is determined only by its number of particles and by the movement of a particle of that region, which is independent of what happens in other regions. At time 0, the regions $\mathcal{R}_{(0,k)}$ and $\mathcal{R}_{(0,-k)}$ are activated in that the origin is infected. It follows that, if with positive probability one of the sites $(0,k)$ and $(0,-k)$ belongs to an infinite open cluster in an independent site percolation process on $\mathbb{Z}^2_k$ where each site is open with probability $p_{\lambda,\gamma}$, then $\PP(S_{\lambda,\gamma})>0$ in the epidemic process. The probability $p_{\lambda,\gamma}$ can be made larger than the critical value for standard site percolation by decreasing $\gamma$ and hence the theorem follows.
\end{proof}

\emph{Remark:} We stress that the reason for working with an initial configuration with fewer particles is that regions are then classified (as open or not) independently. The proof relies only on infection transfer via sites -- that is, not via direct contact between particles -- which justifies the remark after Corollary \ref{corr:super}. As mentioned in Section 1, it is natural to expect that the process is supercritical also for sufficiently small $\lambda$ with $\gamma$ fixed (perhaps $\gamma=\infty$ so that there is no site infection). The infection then survives in that particles remain infected for a long time. However, generalizing the above basic proof to cover also this case would lead to dependencies between regions: we then need that particles in neighboring regions meet in order for the infection to survive and this introduces dependencies in the classification for nearest neighbor and second nearest neighbor regions. Percolation processes with finite dependencies can be dominated by product measure for large enough marginal probabilities, see \cite{LSS}, but our process has complicated dependencies in space/time and we have not been able to sort out necessary estimates and other details in such a comparison.\hfill$\Box$

\begin{proof}[Proof of Theorem \ref{th:sub}]
We first describe a process that dominates our epidemic process in the sense that, if the dominating process dies out, then also the epidemic process dies out. We then show that the dominating process dies out almost surely for large $\lambda$ by in turn dominating it with a branching process.


Let $v_k=v_k(d)$ denote the number of sites within $L_1$-distance $k$ from a given site in $\mathbb{Z}^d$. Consider a process where the number of particles per site is constant over time and equal to $\bar{m}v_k$ -- this is the maximal number of particles that can be located at a site at any given time in the epidemic process. At time 0, the origin and all its particles are infected. The infection then evolves as follows. Each infected particle sends out infectious signals at rate 1. An infectious signal is directed towards a randomly chosen neighbor and the signal causes this neighboring site along with all its $\bar{m}v_k$ particles to become infected. All particles -- both infected and healthy -- also send out reinforcement signals at rate 1, independently of the infectious signals. A reinforcement signal is also directed towards a randomly chosen neighbor and, if there is infection present at the neighboring site -- that is, if either the site or any of the particles on it is infected -- then the infection is reinforced in that all particles and the site become infected. An infected particle recovers at rate $\lambda$ and an infected site is cleared at rate $\gamma$, independently of everything else. This process will be referred to as the maximal load epidemic.

The maximal load epidemic contains more particles than the original process, since the number of particles per site is constant equal to the maximal possible number in the epidemic. This clearly helps the infection. It also makes it easier to work with than the original process in that the number of particles per site is independent of the configuration at neighboring sites. Furthermore, the particles do not move, but affect their neighbors in a similar way as in the epidemic: An infectious signal correspond to an infected particle jumping to a neighboring site in the epidemic process, and a reinforcement signal correspond to a healthy particle jumping onto a site where infection is present and then becoming infected. Apart from containing more particles, the maximal load epidemic dominates the original epidemic in that infected particles send out both infectious signals and reinforcement signals independently, and in that a reinforcement signal causes everything at a site to be infected (rather than the number of infected particles at the site going up by one). It is straightforward to couple the two processes in such a way that the number of infected particles per site at each given time is at least as large in the maximal load epidemic as in the original epidemic, and a site that is contaminated in the original epidemic is also contaminated in the maximal load epidemic.


One way of constructing the maximal load epidemic is as follows. For $x\in\mathbb{R}^d$ and $t\geq 0$, let $\mathcal{P}_{x,t}$ be a family of independent Poisson processes started at time $t$ consisting of:
\begin{itemize}
\item[(i)] one Poisson process with rate 1 for each one of the $\bar{m}v_k$ particles at $x$ -- these govern the times when infectious signals can be sent out from the particles at $x$;
\item[(ii)] one Poisson process with rate $1$ for each one of the particles at the neighboring sites of $x$ -- these govern the times when $x$ is hit by reinforcement signals;
\item[(iii)] one Poisson process with rate $\lambda$ for each particle at $x$ -- these govern the times when the particles at $x$ can recover;
\item[(iv)] one Poisson process for the site $x$ -- this govern times when $x$ can recover.
\end{itemize}
We now define a local infection process at $x$ started at time $t$ by infecting the site $x$ and all its particles and assigning a Poisson family $\mathcal{P}_{x,t}$ to $x$. The particles at $x$ then send out infectious signals, recover and (possibly) get re-infected by reinforcement signals from the neighbors at times determined by the processes in $\mathcal{P}_{x,t}$. Also the site $x$ is cleared and is (possibly) re-infected at times determined by $\mathcal{P}_{x,t}$. Note that infectious signals from the neighbors are not taken into account in defining a local process -- these will instead start new local processes, as described below. The local process terminates after a finite time when the site $x$ is cleared and all its particles have recovered, and the output of the local process is the set $\mathcal{S}_{x,t}$ of infectious signals sent out from $x$ up until that time. An element of $\mathcal{S}_{x,t}$ can be represented as $(y,s)$, where $y$ is the neighbor of $x$ receiving the infectious signal and $s\geq t$ is the time when this happens.

The maximal load epidemic is obtained by starting a local infection process at the origin at time 0. Then, when a neighboring site of the origin receives an infectious signal, a local infection process based on a Poisson family that is independent of $\mathcal{P}_{0,0}$ is started at the neighboring site. This is then iterated: infectious signals start new local infection processes based on Poisson families that are independent of all previous Poisson families. Furthermore, if a local infection process is still running at a site that is hit by a new infectious signal, then the former process is stopped.

Now consider a maximal load epidemic process generated as above, but with the modification that a local infection process is not stopped when its site is hit by a new infectious signal, but continues to generate infectious signals in parallel with the new local process until it (that is, the first local process) terminates by itself. Note that all local processes are based on independent Poisson families, as described above. An infectious signal sent out in a local process after a site has been hit by a new infectious signal has no counterpart in the epidemic, and hence this modified process dominates the maximal load epidemic in the sense that, if an infectious signal is sent out from a given site at a given time in the maximal load epidemic, then the same signal is sent out also in the modified process. As a consequence, if only finitely many infectious signals are sent out in the modified process, then the same is true also for the maximal load epidemic. Furthermore, if only finitely many infectious signals are sent out in the maximal load epidemic, then the process dies out: when the last infectious signal has been sent out, no transfer of infection will occur to healthy sites without infected particles, and (the finitely many) sites that are already infected or have infected particles will be free from infection within finite time.

The infectious signals in the modified process can be described by a branching process as follows. Let the first generation $\mathcal{Z}_1$ consist of $\mathcal{S}_{0,0}$ -- the infectious signals sent out in the initial local infection process. The second generation consists of the infectious signals generated by the local infection processes initiated by the elements in $\mathcal{S}_{0,0}$, that is,
$$
\mathcal{Z}_2=\left\{\mathcal{S}_{y,s}: (y,s)\in\mathcal{S}_{0,0} \right\}.
$$
In general, we set
$$
\mathcal{Z}_n=\left\{\mathcal{S}_{y,s}: (y,s)\in\mathcal{Z}_{n-1}\right\}.
$$
An individual $(y,s)$ in the branching process hence corresponds to an infection signal being sent out to a site $y$ at time $s$. Note that this indeed defines a branching process: each infection signal that is sent out initiates an independent local infection process, and the individuals hence generate independent identically distributed offspring.

It remains to determine the expected offspring $\E[|\mathcal{S}_{0,0}|]$ and show that this is strictly smaller than 1 for large $\lambda$, making the process subcritical. To this end, say that the origin is in a fully infected state if both the site and all particles are infected, and consider a slightly modified local infection process where the origin returns to the fully infected state as soon as an infectious signal is sent out from any of its particles. Clearly the expected number of infectious signals sent out in such a process before it terminates is at least as large as in the original process. Write $p_{\sss \rm{part}}$ for the probability that, starting from a fully infected state, all particles at the origin recover before any infectious signal is sent out or the origin is hit by any reinforcement signal from the neighbors. The $\bar{m}v_k$ particles at the origin send out infectious signals independently at rate 1, and there are $2d\bar{m}v_k$ neighboring particles each sending reinforcement signals independently at rate $1$. Hence the time until an infectious signal is sent out or a reinforcement signal hits the origin is exponentially distributed with parameter $\bar{m}v_k(1+2d)$. The time until all particles recover can be dominated by the sum of the recovery times, which is gamma distributed with parameters $\bar{m}v_k$ and $\lambda$. It follows that
$$
p_{\sss \rm{part}}\geq p_{\sss\rm{part}}^* :=\PP\left(\Gamma(\bar{m}v_k,\lambda)<{\rm{Exp}}(\bar{m}v_k(1+2d))\right).
$$
If an infectious signal is sent out or a reinforcement signal is received, then the origin returns to the fully infected state. The number of attempts that it takes until all particles manage to recover before an infectious signal is sent out or a reinforcement signal is received is stochastically dominated by a Ge($p_{\sss\rm{part}}^*$)-variable. Furthermore, this number clearly dominates the number of infectious signals sent out when all particles have recovered.

If the origin site has been cleared by the time when all its particles have recovered, then the local process terminates. If the origin site is still infected, then the remaining time until it is cleared is exponentially distributed with parameter $\gamma$. However, if the origin is hit by a reinforcement signal before the site is cleared, then it returns to the fully infected state, and infectious signals may again be sent out and estimated as above. Let $p_{\sss\rm{site}}$ denote the probability that, from the moment when all its particles have recovered, also the origin site is cleared before it is hit by a reinforcement signal. Then
$$
p_{\sss\rm{site}}\geq p_{\sss\rm{site}}^* := \PP\left({\rm{Exp}}(\gamma)<{\rm{Exp}}(2d\bar{m}v_k)\right).
$$
Now let $\{X_i\}$ be i.i.d.\ Ge$(p_{\sss\rm{part}}^*)$-distributed random variables and let $N$ be a Ge$(p_{\sss\rm{site}}^*)$-variable, independent of $\{X_i\}$. The number of infectious signals sent out by the origin before the local process terminates is then stochastically dominated by $\sum_{i=1}^N X_i$. Hence
$$
\E[|\mathcal{S}_{0,0}|]\leq \E\left[\sum_{i=1}^N X_i \right]=\E[N]\cdot\E[X_1]=\frac{1-p_{\sss\rm{site}}^*}{p_{\sss\rm{site}}^*}
\cdot\frac{1-p_{\sss\rm{part}}^*}{p_{\sss\rm{part}}^*}.
$$
Clearly $p_{\sss\rm{part}}^*\to 1$ as $\lambda\to\infty$, while $p_{\sss\rm{site}}^*$ does not depend on $\lambda$. If follows that $\E[|\mathcal{S}_{0,0}|]<1$ for large enough $\lambda$, which completes the proof.
\end{proof}

\section*{Acknowledgements}
We thank Fredrik Liljeros for inspiring discussions on cattle-epidemics leading to this work. TB and MD are grateful to Riksbankens jubileumsfond (P12-0705:1) for financial support. FL is grateful to CONICYT-FONDECYT (Proyecto de Postdoctorado N. 3160163 ) for financial support.


\begin{thebibliography}{99}

\bibitem{frogs_shape} Alves, O.S.M., Machado, F.P and Popov, S.Y. (2002): The shape theorem for the frog model, \emph{Ann. Appl. Probab.} \textbf{12}, 533-546.

\bibitem{pt_frogs} Alves, O.S.M., Machado, F.P and Popov, S.Y. (2002): Phase transition for the frog model, \emph{Electron. J. Probab.} \textbf{7}, 1-21.

\bibitem{AB} Andersson, H. and Britton, T. (2000): \emph{Stochastic epidemic models and their statistical analysis}, Springer Lecture Notes in Statistics 151.

\bibitem{AS} Andjel, E. and Schinazi, R. (1996): A complete convergence theorem for an epidemic model, \emph{J. Appl. Probab.} \textbf{33}, 741-748.

\bibitem{vdBGS} van den Berg, J., Grimmett, G. and Schinazi, R. (1998): Dependent random graphs and spatial epidemics, \emph{Ann. Appl. Probab.} \textbf{8}, 317-336.

\bibitem{DRS} Dickman, R., Rolla, L.T. and Sidoravicius, V. (2010): Activated random walkers: facts, conjectures and challenges, \emph{J. Stat. Phys.} \textbf{138}, 126-142.

\bibitem{DN} Durrett, R. and Neuhauser, C. (1991): Epidemics with recovery in D=2, \emph{Ann. Appl. Probab.} \textbf{1}, 189-206.
\bibitem{DBook} Durrett, R. (1995): Ten Lectures on Particle Systems. \emph{St. Flour Lecture Notes. Lecture Notes in Math. 1608}, 97-201, Springer-Verlag, New York.

\bibitem{DBook2}Durrett, R. (1990): A new method for proving the existence of phase transitions. \emph{Spatial Stochastic Processes: a Festschrift in Honor
of Ted Harris on his Seventieth Birthday}(edited by K.S. Alexander and J.C. Watkins),  141-170, Birkhauser, Boston.

\bibitem{Grim} Grimmett, G. (1999): \emph{Percolation}, 2nd edition, Springer.

\bibitem{KS_03} Kesten, H. and Sidoravicius, V. (2003): Branching random walk with catalysts, \emph{Electron. J. Probab.} {\bf 8}, 1-51.

\bibitem{H_15} Hollingsworth, T.D.,  Pulliam, J.R.C, Funk, S., Truscott, J.E., Isham, V.,  and Lloyd, A.L. (2015): Seven challenges for modelling indirect transmission: Vector-borne diseases, macroparasites and neglected tropical diseases. \emph{Epidemics}. {\bf 10}, 16–20.

\bibitem{KS_05} Kesten, H. and Sidoravicius, V. (2005): The spread of a rumor or infection in a moving population, \emph{Ann. Probab.} {\bf 33}, 2402-2462.

\bibitem{KS_08} Kesten, H. and Sidoravicius, V. (2008): A shape theorem for the spread of an infection, \emph{Ann. Math.} {\bf 167}, 701-766.

\bibitem{KS_06} Kesten, H. and Sidoravicius, V. (2006): A phase transition in a model for the spread of an epidemic, \emph{Illinois J. Math.} {\bf 50}, 547-634.

\bibitem{Kul} Kuulasmaa, K. (1982): The spatial general epidemic model and locally dependent random graphs, \emph{J. Appl. Probab.} \textbf{19}, 745-758.

\bibitem{LSS} Liggett, T. M. and Schonmann, R. H. and Stacey, A. M. (1997): Domination by product measures, \emph{Ann. Probab.} \textbf{25},  \textbf{1},  71-95.

\bibitem{Mol} Mollison, D. (1977): Spatial models for ecological and epidemic spread, \emph{J. R. Stat. Soc. Ser. B Stat. Methodol.} \textbf{39}, 283-326.

\bibitem{frogs_surv} Popov, S.Y. (2003): Frogs and some other interacting random walk models, \emph{Discrete Math. Theor. Comput. Sci.} \textbf{AC}, 277-288.

\bibitem{RS} Rolla, L.T. and Sidoravicius, V. (2012): Absorbing-state phase transition for driven-dissipative stochastic dynamics on $\mathbb{Z}$, \emph{Invent. Math.} \textbf{188}, 127-150.

\bibitem{ST} Stauffer, A. and Taggi, L. (2015): Critical density of activated random walks on $\mathbb{Z}^d$ and general graphs, ArXiV:1512.02397.

\end{thebibliography}
\end{document}